\documentclass[11pt]{amsart}
\usepackage{amsmath,amssymb,amsthm,booktabs,array}
\usepackage[margin=1.1in]{geometry}
\usepackage{hyperref}

\newtheorem{theorem}{Theorem}[section]
\newtheorem{lemma}[theorem]{Lemma}
\newtheorem{proposition}[theorem]{Proposition}

\theoremstyle{definition}
\newtheorem{definition}[theorem]{Definition}
\newtheorem{example}[theorem]{Example}
\theoremstyle{remark}
\newtheorem{remark}[theorem]{Remark}

\newcommand{\F}{\mathbb{F}}
\newcommand{\GP}{\mathbb{GP}}
\newcommand{\GH}{\mathbb{GH}}
\newcommand{\wt}{\mathrm{wt}}
\newcommand{\supp}{\mathrm{supp}}
\newcommand{\cC}{\mathcal{C}}
\newcommand{\cH}{\mathcal{H}}
\newcommand{\cI}{\mathcal{I}}

\title[Minimal $p$-ary codes from hierarchical posets]{Optimal and minimal $p$-ary linear codes from generalized order ideals of hierarchical posets}
\author{Rumi Melih Pelen}
\address{Department of Mathematics and Statistics, University of South Florida, Tampa, FL, USA}
\email{rmpelen@usf.edu}

\begin{document}

\begin{abstract}
Hyun, Kim, Wu and Yue constructed optimal and minimal binary linear codes from order ideals of hierarchical posets with two levels. Two different generalizations of the underlying antichain (simplicial complex) setting to odd characteristic are known: down-sets of $\F_p^n$ under the componentwise order, and support-closed subsets of $\F_q^m$. No generalization of the poset setting itself has appeared. We introduce generalized order ideals of a poset of order $p-1$, obtained by attaching multiplicities in $\{0,\dots,p-1\}$ to the elements of a poset, and study the two natural notions of order ideal that arise for hierarchical posets with two levels. Whenever the ideal meets the upper level, the resulting defining sets are neither down-sets nor support-closed. We determine the weight distributions of the associated complement codes, exhibit a family of Griesmer codes in which the upper element carries an arbitrary multiplicity, and, via the characteristic function of a generalized order ideal, obtain an infinite family of minimal $p$-ary codes of length $p^n-1$ and dimension $n+1$ violating the Ashikhmin--Barg condition.
\end{abstract}

\maketitle

\section{Introduction}

Let $p$ be a prime and $D\subseteq\F_p^n$. The defining-set construction of Ding and Niederreiter \cite{DN07} associates with $D$ the linear code $\cC_D=\{(u\cdot x)_{x\in D}:u\in\F_p^n\}$, and a large part of the recent literature on few-weight, optimal and minimal linear codes consists of choosing $D$ so that the weight distribution of $\cC_D$ can be computed. Minimal codes, in which the support of no nonzero codeword contains the support of a linearly independent codeword, determine the access structure of the secret sharing schemes built on them \cite{YD06}. The Ashikhmin--Barg condition $w_{\min}/w_{\max}>(p-1)/p$ is sufficient for minimality \cite{AB98}, and constructing infinite families of minimal codes that violate it has been an active problem since the first binary examples of Chang and Hyun \cite{CH18} and Ding, Heng and Zhou \cite{DHZ18}; see \cite{HDZ18,BB19,XQ19,MQRT20,TQLZ21} for the odd characteristic case and for the geometric characterization of minimal codes.

Chang and Hyun \cite{CH18} took $D$ to be the complement of a simplicial complex of $\F_2^n$, and Hyun, Lee and Lee \cite{HLL20} developed the associated generating function technique. Hyun, Kim, Wu and Yue \cite{HKWY20} then observed that a simplicial complex is the set of order ideals of an antichain, replaced the antichain by an arbitrary poset on $[n]$, and worked out the case of hierarchical posets with two levels. They obtained optimal binary codes from complements of families of order ideals, and, using the characteristic function of a family of order ideals, infinite families of minimal binary codes violating the Ashikhmin--Barg condition. Wu, Lu, Cao and Qin \cite{WLQC23} later settled the minimality question for all families of order ideals of two-level hierarchical posets by geometric methods.

In odd characteristic, the antichain case has been generalized in two incompatible ways. Hyun, Kim and Na \cite{HKN19} order $\F_p^n$ componentwise, with $0<1<\dots<p-1$, and call a down-set of this lattice a simplicial complex of $\F_p^n$; the resulting codes are non-projective, and weight distributions were obtained for down-sets generated by one maximal element of support size at most two \cite{HKN19,WH20,SL21,MHL26}. Hu, Xu, Li, Zeng, Wang and Tang \cite{Hu24} instead call a subset of $\F_q^m$ a simplicial complex if it is closed under taking vectors of smaller support; such a set is a union of coordinate subspaces, the resulting codes are projective, and their weight distributions were determined completely, the Griesmer codes among them being Solomon--Stiffler codes. The sets of bounded Hamming weight used in \cite{HDZ18,BB19} and \cite[Section IV]{MQRT20} are unions of coordinate subspaces, hence simplicial complexes in the sense of \cite{Hu24}, so the first nonbinary families violating the Ashikhmin--Barg condition also belong to the antichain setting. The poset construction itself has not been carried to odd characteristic; the survey \cite{survey24} still treats posets only over $\F_2$.

The purpose of this paper is to fill this gap for hierarchical posets with two levels. We attach to every element of a poset $\mathbb{P}$ on $[n]$ a multiplicity in $\{0,\dots,p-1\}$, so that sub-multisets of the resulting multiset $[n]_{p-1}$ correspond to vectors of $\F_p^n$, and we call the outcome a generalized poset of order $p-1$. When $\mathbb{P}$ is an antichain, the generalized order ideals are exactly the down-sets of \cite{HKN19}, and those with all multiplicities equal to $p-1$ are exactly the simplicial complexes of \cite{Hu24}; both known frameworks are therefore the antichain case of ours. For a poset with comparable elements there are two natural notions of generalized order ideal, according to whether an element lying below a chosen element must be taken with full multiplicity $p-1$ (saturated) or merely with positive multiplicity (unsaturated). For the two-level hierarchical poset $\mathbb{H}(m,n)$ both notions produce defining sets that are neither down-sets nor unions of coordinate subspaces (Remark~\ref{rem:notdownset}), so the codes obtained here are outside both existing frameworks.

Our results are as follows. Restricting to multiplicities in $\{0,p-1\}$, we determine the weight distributions of the complement codes for both notions (Theorems~\ref{thm:complement-S} and~\ref{thm:complement-U}); the saturated notion gives at most six nonzero weights independently of the parameters, and the parity condition on $\wt(v)$ that governs the binary weights is replaced by the linear condition $\sum_{i\le m}v_i=0$ in $\F_p$. When the upper level is a single element, taken with any multiplicity $t\in\{1,\dots,p-1\}$, the saturated complement codes form a family of non-projective Griesmer codes with at most five nonzero weights, minimal exactly when $2\le m\le n-2$ (Theorem~\ref{cor:griesmer}), which we identify with a Belov-type anticode construction (Remark~\ref{rem:belov}). We then consider the code $\cC_f$ generated by the simplex-type code and the characteristic function $f$ of a generalized order ideal, the $p$-ary analogue of the second construction of \cite{HKWY20}. Its weight distribution follows from the same character sums (Lemma~\ref{lem:Cf-weight}), and we prove that for both notions $\cC_f$ is minimal exactly when the generalized order ideal has full support and the lower level has at least two elements, with the additional restriction $m\le n-2$ in the saturated case (Theorems~\ref{thm:main} and~\ref{thm:U}, Propositions~\ref{prop:necessary} and~\ref{prop:nonminimal}). These codes have length $p^n-1$, dimension $n+1$, and violate the Ashikhmin--Barg condition (Theorem~\ref{thm:main}, Proposition~\ref{prop:U-AB}). The codes $\cC_f$ are instances of the characteristic-function construction of Mesnager, Qi, Ru and Tang \cite{MQRT20}, where the sets treated in odd characteristic are unions of subspaces with pairwise trivial intersections. Their Theorem 3.11 shows that the union of two complementary subspaces of different dimensions yields minimal codes violating the Ashikhmin--Barg condition for every $p$. The set behind Theorem~\ref{thm:main} is obtained from that set by replacing the second subspace with a coset not through the origin; it is not closed under $x\mapsto-x$. The two families have the same length, dimension and minimum distance but different weight distributions, so the codes are inequivalent (Remark~\ref{rem:MQRT}), and ours includes the equal-dimension case for all $p$. The minimality proofs use the geometric characterization of minimal codes as cutting sets \cite{ABN,TQLZ21}, in the spirit of the direct support arguments of \cite{BB19}, rather than the Walsh spectrum criterion of \cite{HDZ18}, whose case analysis for odd $p$ is noted in \cite{MQRT20} to be laborious; the saturated case reduces to a cardinality bound, the unsaturated case to an explicit spanning argument.

The paper is organized as follows. Section~\ref{sec:prelim} recalls the two minimality criteria we use and introduces generalized posets, generalized order ideals and the associated generating functions. Section~\ref{sec:hier} describes the generalized order ideals of $\GH(m,n)$ and computes the character sums on which everything depends. Section~\ref{sec:complement} treats the complement codes and Section~\ref{sec:Cf} the codes $\cC_f$. Section~\ref{sec:ternary} removes the restriction on multiplicities for $p=3$, where every factor of the generating function simplifies, and determines the weight distributions of the ternary codes from arbitrary generalized order ideals of $\GH(m,n)$; this contains the ternary cases of \cite{HKN19,PL22}. Section~\ref{sec:conclusion} discusses what changes for intermediate multiplicities when $p\ge5$ and lists open problems.

Throughout, $p$ is an odd prime, $\zeta=\zeta_p$ a primitive complex $p$-th root of unity, and $[n]=\{1,\dots,n\}$. For $u,x\in\F_p^n$ we write $u\cdot x=\sum_i u_ix_i$, and for a subset $D\subseteq\F_p^n$ we consider the code
\begin{equation}\label{eq:CD}
\cC_D=\{c_D(u)=(u\cdot x)_{x\in D}:u\in\F_p^n\}.
\end{equation}

\section{Preliminaries}\label{sec:prelim}

\subsection{Minimal codes}
A nonzero codeword $c$ of a linear code $\cC$ over $\F_p$ is minimal if $\supp(c')\subseteq\supp(c)$ with $c'\in\cC$ implies $c'\in\F_p c$; $\cC$ is minimal if all its nonzero codewords are. We use the following two facts.

\begin{lemma}[Ashikhmin--Barg]\label{lem:AB}
If $w_{\min}/w_{\max}>(p-1)/p$, then $\cC$ is minimal.
\end{lemma}

\begin{lemma}[geometric criterion, \cite{ABN,TQLZ21}]\label{lem:geom}
Let $D\subseteq\F_p^k$ span $\F_p^k$. Then $\cC_D$ is minimal if and only if for every nonzero $u\in\F_p^k$ the set $D\cap u^\perp$ spans the hyperplane $u^\perp=\{x:u\cdot x=0\}$.
\end{lemma}
\begin{proof}
Since $D$ spans $\F_p^k$, $u\mapsto c_D(u)$ is injective, and $\supp(c_D(v))\subseteq\supp(c_D(u))$ holds if and only if $D\cap u^\perp\subseteq v^\perp$. Hence $c_D(u)$ is not minimal if and only if $D\cap u^\perp$ is contained in $v^\perp$ for some $v\notin\langle u\rangle$, that is, if and only if $D\cap u^\perp$ lies in a hyperplane of $\F_p^k$ other than $u^\perp$, which is the case exactly when it does not span $u^\perp$. In \cite{ABN,TQLZ21} this is stated for projective codes, as the characterization of minimal codes by cutting blocking sets; the argument does not use that the elements of $D$ are pairwise independent.
\end{proof}

\subsection{Generalized posets of order $p-1$}

Let $\mathbb{P}=([n],\preceq)$ be a poset. The generalized poset of order $p-1$ on $\mathbb{P}$, denoted $\GP$, is the multiset $[n]_{p-1}$ in which every $i\in[n]$ has multiplicity $p-1$, together with $\preceq$. A sub-multiset $J$ of $[n]_{p-1}$ is recorded by its multiplicity vector $(v_1,\dots,v_n)\in\{0,\dots,p-1\}^n$, and we identify $J$ with the vector $v\in\F_p^n$. Thus sub-multisets of $[n]_{p-1}$ correspond bijectively to $\F_p^n$, and $J\subseteq J'$ means $v_i\le v'_i$ for all $i$. The support $\bar J\subseteq[n]$ of $J$ is $\{i:v_i\ne0\}$.

\begin{definition}\label{def:ideal}
Let $J$ be a sub-multiset of $[n]_{p-1}$.
\begin{enumerate}
\item[(U)] $J$ is an \emph{unsaturated generalized order ideal} if $\bar J$ is an order ideal of $\mathbb{P}$.
\item[(S)] $J$ is a \emph{saturated generalized order ideal} if $\bar J$ is an order ideal of $\mathbb{P}$ and, whenever $j\in\bar J$ and $i\prec j$, the multiplicity of $i$ in $J$ equals $p-1$.
\end{enumerate}
For a generalized order ideal $I$ (of either type) we write $I(\GP)$ for the set of generalized order ideals of the same type contained in $I$, viewed as a subset of $\F_p^n$.
\end{definition}

The word ``order'' in ``generalized poset of order $p-1$'' refers to the common multiplicity $p-1$, not to the cardinality of the poset. The following observation records that both known odd-characteristic frameworks are the antichain case of Definition~\ref{def:ideal}.

\begin{proposition}\label{prop:recover}
\begin{enumerate}
\item If $\mathbb{P}$ is an antichain, the two notions of Definition~\ref{def:ideal} coincide, $I(\GP)$ is the down-set $\langle v\rangle$ of $\F_p^n$ generated by the multiplicity vector $v$ of $I$ in the componentwise order of \cite{HKN19}, and every down-set of $\F_p^n$ is $\cI(\GP)$ for the family $\cI$ of its maximal elements.
\item If $\mathbb{P}$ is an antichain and every multiplicity of $I$ lies in $\{0,p-1\}$, then $I(\GP)=\F_p^{\bar I}\times0$ is a coordinate subspace, and $\cI(\GP)$ for a family $\cI$ of such ideals is a support-closed set (a simplicial complex of $\F_p^n$ in the sense of \cite{Hu24}); every support-closed set arises this way.
\item For $p=2$ the multiplicities are $0$ or $1$, both notions coincide, and Definition~\ref{def:ideal} reduces to the order ideals of $\mathbb{P}$ used in \cite{HKWY20}.
\end{enumerate}
\end{proposition}
\begin{proof}
In an antichain no element lies below another, so the saturation condition is vacuous, $\bar J$ is always an order ideal, and $J\subseteq I$ means $J\preceq v$ componentwise. Part (2) is the special case where the box $\langle v\rangle$ is a coordinate subspace, and a union of coordinate subspaces is exactly a support-closed set. Part (3) is immediate.
\end{proof}

\subsection{Generating functions and the weight formula}

For $X\subseteq\F_p^n$ let $\cH_X(x_1,\dots,x_n)=\sum_{u\in X}\prod_i x_i^{u_i}\in\mathbb{Z}[x_1,\dots,x_n]$. Inclusion--exclusion gives, for $\cI=\{I_1,\dots,I_k\}$ and $\cI(\GP)=\bigcup_j I_j(\GP)$,
\begin{equation}\label{eq:incl-excl}
\cH_{\cI(\GP)}=\sum_{\emptyset\ne S\subseteq\cI}(-1)^{|S|+1}\cH_{\bigcap_{I\in S}I(\GP)} .
\end{equation}

\begin{lemma}\label{lem:weight}
Let $D=\F_p^n\setminus\cI(\GP)$. For $u\ne0$,
\[
\wt(c_D(u))=\frac{p-1}{p}\,|D|+\frac1p\sum_{y\in\F_p^*}\cH_{\cI(\GP)}(\zeta^{u_1y},\dots,\zeta^{u_ny}).
\]
\end{lemma}
\begin{proof}
$\wt(c_D(u))=|D|-\frac1p\sum_{y\in\F_p}\sum_{x\in D}\zeta^{y\,u\cdot x}$, and $\sum_{x\in D}\zeta^{y\,u\cdot x}=-\sum_{x\in\cI(\GP)}\zeta^{y\,u\cdot x}$ for $y\ne0$, $u\ne0$.
\end{proof}

\section{Hierarchical posets with two levels}\label{sec:hier}

Let $1\le m\le n$, $U=[m]$, $V=[n]\setminus[m]$, and let $\GH(m,n)$ be the generalized poset of order $p-1$ on the hierarchical poset $\mathbb{H}(m,n)$ in which $U$ and $V$ are antichains and $i\prec j$ for all $i\in U$, $j\in V$. Every generalized order ideal is $I=A\cup B$ with $A\subseteq[m]_{p-1}$, $B\subseteq V_{p-1}$, and either $B=\emptyset$ or ($B\ne\emptyset$ and $\bar A=[m]$).

In this paper we restrict to multiplicities in $\{0,p-1\}$, so $A=(\bar A)_{p-1}$, $B=(\bar B)_{p-1}$, and set $a=|\bar A|$, $b=|\bar B|$. Write $u=(v,w)$ with $v\in\F_p^m$, $w\in\F_p^{n-m}$, $k=\wt(v)$, $\sigma=\sum_{i\le m}v_i\in\F_p$, $\delta=[w_{\bar B}=0]$, and $Q=(p-1)p^{n-1}$.

\begin{lemma}\label{lem:ideal-sets}
Let $I=A\cup B$ as above.
\begin{enumerate}
\item If $B=\emptyset$ then, for both notions, $I(\GP)=\F_p^{\bar A}\times0$, of size $p^a$.
\item If $B\ne\emptyset$ (so $\bar A=[m]$), then
\begin{align*}
I(\GP)&=(\F_p^m\times0)\ \cup\ \big((\F_p^*)^m\times(\F_p^{\bar B}\setminus0)\big) &&\text{(U)},\quad |I(\GP)|=p^m+(p-1)^m(p^b-1),\\
I(\GP)&=(\F_p^m\times0)\ \cup\ \big(\{-1\}^m\times(\F_p^{\bar B}\setminus0)\big) &&\text{(S)},\quad |I(\GP)|=p^m+p^b-1,
\end{align*}
where $\F_p^{\bar B}$ denotes the vectors of $\F_p^{n-m}$ supported on $\bar B$ and $-1=p-1$.
\end{enumerate}
\end{lemma}

\begin{remark}\label{rem:notdownset}
For $B\ne\emptyset$ the set $I(\GP)$ is not a down-set of $\F_p^n$ in the componentwise order (it contains $(-1,\dots,-1,x'')$ but not $(0,\dots,0,x'')$ for $x''\ne0$), and it is not support-closed. Hence neither the framework of Hyun--Kim--Na nor that of Hu et al.\ contains these defining sets. In case (U) the set $I(\GP)$ is closed under $\F_p^*$-scaling, so $\cC_D$ is a $(p-1)$-fold repetition of a projective code; in case (S) it is not, and $\cC_D$ is non-projective.
\end{remark}

\begin{lemma}[character sums]\label{lem:charsum}
Let $B\ne\emptyset$, $y\in\F_p^*$, $u=(v,w)\ne0$.
\begin{align*}
\text{(U)}\quad&\sum_{x\in I(\GP)}\zeta^{y\,u\cdot x}=p^m[v=0]+(-1)^k(p-1)^{m-k}\,(p^b\delta-1),\\
\text{(S)}\quad&\sum_{x\in I(\GP)}\zeta^{y\,u\cdot x}=p^m[v=0]+\zeta^{-y\sigma}\,(p^b\delta-1).
\end{align*}
\end{lemma}
\begin{proof}
$\sum_{x'\in(\F_p^*)^m}\zeta^{y\,v\cdot x'}=\prod_{i\le m}(p[v_i=0]-1)$, $\sum_{x''\in\F_p^{\bar B}\setminus0}\zeta^{y\,w\cdot x''}=p^b\delta-1$, and $v\cdot(-1,\dots,-1)=-\sigma$.
\end{proof}

\section{Weight distributions of the complement codes}\label{sec:complement}

\begin{theorem}[$B=\emptyset$]\label{thm:B-empty}
Let $D=\F_p^n\setminus(\F_p^{\bar A}\times0)$ with $1\le a\le n-1$. Then $\cC_D$ is a two-weight Griesmer $[p^n-p^a,\,n,\,(p-1)(p^{n-1}-p^{a-1})]$ code with weights $(p-1)(p^{n-1}-p^{a-1})$ (frequency $p^n-p^{n-a}$) and $(p-1)p^{n-1}$ (frequency $p^{n-a}-1$). For $a=0$ (the empty ideal) $D=\F_p^n\setminus\{0\}$ and $\cC_D$ is the one-weight $[p^n-1,\,n,\,(p-1)p^{n-1}]$ code, the $(p-1)$-fold repetition of the simplex code.
\end{theorem}

The same one-coordinate principle as in Theorem~\ref{cor:griesmer} below applies to an antichain: one coordinate of the generalized order ideal may carry any multiplicity.

\begin{proposition}\label{prop:antichain-t}
Let $1\le a\le n-2$, $1\le t\le p-1$, and $D=\F_p^n\setminus(\F_p^a\times\{0,\dots,t\}\times0)$. Then $\cC_D$ is a Griesmer code with at most three nonzero weights, with parameters $[p^n-(t+1)p^a,\,n,\,d]_p$, $d=(p-1)p^{n-1}-(p-1)(t+1)p^{a-1}$, and weights
\[
\begin{array}{ll}
d & p^n-p^{n-a}\\
d+(p-t-1)p^{a-1} & (p-1)p^{n-a-1}\\
(p-1)p^{n-1} & p^{n-a-1}-1 .
\end{array}
\]
For $t=p-1$ the first two rows coincide and the code is the two-weight code of Theorem~\ref{thm:B-empty} with $a+1$ in place of $a$.
\end{proposition}
\begin{proof}
Write $u=(u_A,u_{a+1},u')$. The character sum over $\F_p^a\times\{0,\dots,t\}\times0$ is $p^a[u_A=0]\sum_{j=0}^t\zeta^{yu_{a+1}j}$, whose sum over $y\in\F_p^*$ is $p^a[u_A=0]\big((p-1)(t+1)[u_{a+1}=0]+(p-t-1)[u_{a+1}\ne0]\big)$; Lemma~\ref{lem:weight} gives the three weights. For the Griesmer bound, $d/p^i$ is an integer for $i\le a-1$; for $i=a$, $d/p^a=(p-1)p^{n-1-a}-(t+1)+(t+1)/p$, so $\lceil d/p^a\rceil=(p-1)p^{n-1-a}-t$; and for $a<i\le n-1$, $\lceil d/p^i\rceil=(p-1)p^{n-1-i}$, since $(p-1)(t+1)/p^{i-a+1}<1$. Hence $\sum_{i=0}^{n-1}\lceil d/p^i\rceil=(p^n-1)-(t+1)(p^a-1)-t=p^n-(t+1)p^a$. The case $a=0$ is \cite[Theorem 4.1]{HKN19} and the case $t=p-1$ is contained in \cite{Hu24}.
\end{proof}

\begin{theorem}[case (S), $B\ne\emptyset$]\label{thm:complement-S}
Let $I=[m]_{p-1}\cup(\bar B)_{p-1}$, $1\le b\le n-m$, with the saturated notion, and $D=\F_p^n\setminus I(\GP)$. Then $\cC_D$ is an $[p^n-p^m-p^b+1,\,n]$ code with at most six nonzero weights, given in Table~\ref{tab:complement-S}, where $c_1=p^{n-m-b}$ and $c_0=p^{n-m}-p^{n-m-b}$.
\end{theorem}

\begin{table}[h]
\centering
\begin{tabular}{lll}
\toprule
condition on $u=(v,w)$ & weight & frequency\\
\midrule
$v=0,\ \delta=1$ & $Q$ & $c_1-1$\\
$v=0,\ \delta=0$ & $(p-1)(p^{n-1}-p^{b-1})$ & $c_0$\\
$v\ne0,\ \sigma=0,\ \delta=1$ & $(p-1)(p^{n-1}-p^{m-1})$ & $(p^{m-1}-1)c_1$\\
$v\ne0,\ \sigma=0,\ \delta=0$ & $(p-1)(p^{n-1}-p^{m-1}-p^{b-1})$ & $(p^{m-1}-1)c_0$\\
$\sigma\ne0,\ \delta=1$ & $(p-1)(p^{n-1}-p^{m-1})-p^b+1$ & $(p^m-p^{m-1})c_1$\\
$\sigma\ne0,\ \delta=0$ & $(p-1)(p^{n-1}-p^{m-1}-p^{b-1})+1$ & $(p^m-p^{m-1})c_0$\\
\bottomrule
\end{tabular}
\caption{Weight distribution in Theorem~\ref{thm:complement-S}.}
\label{tab:complement-S}
\end{table}

\begin{proof}
By Lemma~\ref{lem:charsum}, $\sum_{y\ne0}\zeta^{-y\sigma}=p[\sigma=0]-1$, so Lemma~\ref{lem:weight} gives
$\wt(c_D(u))=\frac{p-1}{p}|D|+\frac1p\big((p-1)p^m[v=0]+(p[\sigma=0]-1)(p^b\delta-1)\big)$. The six cases follow; frequencies count $v$ and $w$ separately.
\end{proof}

The restriction to multiplicities in $\{0,p-1\}$ can be dropped on a single coordinate at no cost, since a one-coordinate box is just a set of $t$ points. This gives the following family, which contains the case $b=1$ of Theorem~\ref{thm:complement-S} as $t=p-1$.

\begin{theorem}\label{cor:griesmer}
Let $2\le m\le n-1$, $n\ge3$, $1\le t\le p-1$, and let $I=[m]_{p-1}\cup\{(m+1)^{(t)}\}$ be the saturated generalized order ideal of $\GH(m,n)$ in which the single upper element $m+1$ has multiplicity $t$, so that $I(\GP)=(\F_p^m\times0)\cup\{(-1,\dots,-1,j,0,\dots,0):1\le j\le t\}$ and $D=\F_p^n\setminus I(\GP)$. Then $\cC_D$ is a Griesmer code with parameters
\[
\big[p^n-p^m-t,\ n,\ (p-1)(p^{n-1}-p^{m-1})-t\big]_p .
\]
Writing $d$ for the minimum distance and $c=p^{n-m-1}$, its nonzero weights and frequencies are
\[
\begin{array}{ll}
d & (p-1)c\,\big((p-t+1)p^{m-1}-1\big)\\
d+1 & (p-1)c\,t\,p^{m-1}\\
(p-1)(p^{n-1}-p^{m-1}) & (p^{m-1}-1)c\\
(p-1)p^{n-1}-t & (p-1)c\\
(p-1)p^{n-1} & c-1
\end{array}
\]
For $t\ge2$ these are five distinct weights if $m\le n-2$ and four if $m=n-1$; for $t=1$ the second and third rows coincide, giving four and three distinct weights respectively. In all cases $\cC_D$ is minimal if and only if $m\le n-2$.
\end{theorem}
\begin{proof}
For $u=(v,w_1,w')\ne0$ with $\sigma=\sum_{i\le m}v_i$, the character sum over $I(\GP)$ is $p^m[v=0]+\sum_{j=1}^t\zeta^{-y(\sigma-jw_1)}$, and summing over $y\in\F_p^*$ gives $(p-1)p^m[v=0]+\sum_{j=1}^t(p[\sigma=jw_1]-1)$. If $w_1=0$ the last sum is $t(p[\sigma=0]-1)$; if $w_1\ne0$ it is $p[\sigma\ne0,\ \sigma/w_1\in\{1,\dots,t\}]-t$. Lemma~\ref{lem:weight} with $|D|=p^n-p^m-t$ gives
\[
\wt(c_D(u))=W_0+(p-1)p^{m-1}[v=0]+X,\qquad W_0=(p-1)(p^{n-1}-p^{m-1})-t,
\]
where $X=t[\sigma=0]$ if $w_1=0$ and $X=[\sigma\ne0,\ \sigma/w_1\in\{1,\dots,t\}]$ if $w_1\ne0$. Enumerating $v=0$; $v\ne0,\sigma=0$; $\sigma\ne0$, each with $w_1=0$ or $w_1\ne0$, and noting that for fixed $\sigma\ne0$ exactly $t$ values of $w_1\ne0$ satisfy $\sigma/w_1\in\{1,\dots,t\}$, yields the table; $w'$ ranges over $\F_p^{n-m-1}$, contributing the factor $c$, and $u=0$ is removed from the first row.

For the Griesmer bound, $d=(p-1)p^{n-1}-(p-1)p^{m-1}-t$. For $1\le i\le m-1$, $\lceil d/p^i\rceil=(p-1)(p^{n-1-i}-p^{m-1-i})$ since $0<t/p^i<1$; for $m\le i\le n-2$, $\lceil d/p^i\rceil=(p-1)p^{n-1-i}$ since $0<((p-1)p^{m-1}+t)/p^i\le (p-1)/p+(p-1)/p^m<1$ by $m\ge2$; and $\lceil d/p^{n-1}\rceil=p-1$ since $(p-1)p^{m-1}+t\le(p-1)(p^{n-2}+1)<p^{n-1}$ for $n\ge3$. Hence $\sum_{i=0}^{n-1}\lceil d/p^i\rceil=(p^n-1)-(p^m-1)-t=|D|$.

If $m\le n-2$, then $p\,d-(p-1)(p-1)p^{n-1}=(p-1)(p^{n-1}-p^m)-pt\ge(p-1)^2p^{n-2}-pt>0$, so the Ashikhmin--Barg condition holds and $\cC_D$ is minimal. If $m=n-1$, the fourth row has weight $(p-1)p^{n-1}-t=|D|$, so those codewords have full weight and $\cC_D$ is not minimal.
\end{proof}

\begin{remark}[geometric description]\label{rem:belov}
Let $\Pi$ be the projective $(m-1)$-space $\{x''=0\}$ and let $L$ be the projective line through $\langle(1,\dots,1,0,\dots,0)\rangle\in\Pi$ and $\langle e_{m+1}\rangle$. Regarded as a multiset of points of $PG(n-1,p)$, the defining set $D$ of Theorem~\ref{cor:griesmer} consists of every point outside $\Pi$ with multiplicity $p-1$, except $t$ of the $p-1$ points of $L\setminus\Pi$, which have multiplicity $p-2$; the points of $\Pi$ do not occur. Indeed, of the $p-1$ representatives $\lambda(-1,\dots,-1,c)$ of a point of $L\setminus\Pi$, only $\lambda=1$ lies in $I(\GP)$. Thus $\cC_D$ is the Belov-type (anticode) Griesmer code obtained from the $(p-1)$-fold simplex code by deleting $p-1$ copies of an $(m-1)$-dimensional projective subspace and $t$ points of a line meeting it, in the same way that the codes of Hyun--Kim--Na coincide in parameters with Belov-type codes. The novelty is only that this multiset arises from a poset; the code is non-projective since its columns are repeated.
\end{remark}

\begin{proposition}\label{prop:complement-minimal}
In Theorem~\ref{thm:complement-S} with $2\le m\le n-2$ and $b\le n-m-1$, $\cC_D$ satisfies the Ashikhmin--Barg condition and is therefore minimal.
\end{proposition}
\begin{proof}
Since $b\le n-m-1$, the first row of Table~\ref{tab:complement-S} has frequency $c_1-1\ge p-1>0$, so $w_{\max}=Q$. Comparing the remaining rows, the smallest weight is that of the fifth row, $w_{\min}=Q-(p-1)p^{m-1}-p^b+1$: it is below the fourth row by $p^{b-1}-1\ge0$ and below the second row by $(p-1)p^{m-1}+p^{b-1}-1>0$. Then
\[
p\,w_{\min}-(p-1)w_{\max}=(p-1)(p^{n-1}-p^m)-p^{b+1}+p>0,
\]
because $m\le n-2$ gives $(p-1)(p^{n-1}-p^m)\ge(p-1)^2p^{n-2}$ and $b+1\le n-m\le n-2$ gives $p^{b+1}\le p^{n-2}$.
\end{proof}

\begin{theorem}[case (U), $B\ne\emptyset$]\label{thm:complement-U}
With the unsaturated notion and $(m,b)\ne(1,n-1)$, $|D|=p^n-p^m-(p-1)^m(p^b-1)$, $\dim\cC_D=n$, and, for $u\ne0$ with $T_k=(-1)^k(p-1)^{m-k}$,
\[
\wt(c_D(u))=\frac{p-1}{p}\Big(|D|+p^m[k=0]+T_k(p^b\delta-1)\Big),
\]
with frequency $\binom mk(p-1)^k c_\delta$ (minus one for $(k,\delta)=(0,1)$). The code has at most $2(m+1)$ weights; for $m=1$ and $b\le n-2$ it is a three-weight code of length $p^n-(p-1)p^b-1$ with weights $(p-1)p^{n-1}-(p-1)^2p^{b-1}$, $(p-1)(p^{n-1}-p^b)$, $(p-1)p^{n-1}$. In the excluded case $(m,b)=(1,n-1)$ one has $D=\{0\}\times(\F_p^{n-1}\setminus0)$, the map $u\mapsto c_D(u)$ has kernel $\F_p e_1$, and $\cC_D$ is the one-weight $[p^{n-1}-1,\,n-1,\,(p-1)p^{n-2}]$ code.
\end{theorem}

\begin{remark}
Two incomparable ideals $I_j=[m]_{p-1}\cup(\bar B_j)_{p-1}$ can be treated by \eqref{eq:incl-excl} with $I_1\cap I_2=[m]_{p-1}\cup(\bar B_1\cap\bar B_2)_{p-1}$; the computation is identical. Computer search ($p=3$, $n\le4$) produced only minimal codes satisfying the Ashikhmin--Barg condition from the complement construction.
\end{remark}

\section{Minimal codes from characteristic functions of generalized order ideals}\label{sec:Cf}

Let $f:\F_p^n\to\F_p$ be the characteristic function of $I(\GP)\setminus\{0\}$ and
\begin{equation}\label{eq:Cf}
\cC_f=\{c(s,u)=(s f(x)+u\cdot x)_{x\in\F_p^n\setminus0}:\ s\in\F_p,\ u\in\F_p^n\}.
\end{equation}
This is the code $\cC_{D'}$ of \eqref{eq:CD} with $D'=\{(f(x),x):x\ne0\}\subseteq\F_p^{n+1}$; it is the $p$-ary analogue of construction (5.5) of Hyun--Kim--Wu--Yue and an instance of the characteristic-function construction of Mesnager--Qi--Ru--Tang.

\begin{lemma}\label{lem:Cf-weight}
Let $I=[m]_{p-1}\cup(\bar B)_{p-1}$, $B\ne\emptyset$. Then $\wt(c(0,u))=Q$ for $u\ne0$, $\wt(c(s,0))=|I(\GP)|-1$ for $s\ne0$, and for $s\ne0$, $u\ne0$,
\[
\wt(c(s,u))=Q-1+p^m[v=0]+\Phi,\qquad
\Phi=\begin{cases}(-1)^k(p-1)^{m-k}(p^b\delta-1)&\text{(U)},\\[2pt](p^b\delta-1)\big([\sigma=0]-[\sigma=s]\big)&\text{(S)}.\end{cases}
\]
\end{lemma}
\begin{proof}
For $s\ne0$, $\wt(c(s,u))=p^n-1-N_{\mathrm{all}}(0)+N_I(0)-N_I(-s)$, where $N_I(c)=|\{x\in I(\GP):u\cdot x=c\}|$ and $N_{\mathrm{all}}(0)=|\{x:u\cdot x=0\}|$. Since $N_I(0)-N_I(-s)=\frac1p\sum_{y\ne0}(1-\zeta^{ys})\sum_{x\in I(\GP)}\zeta^{y\,u\cdot x}$, Lemma~\ref{lem:charsum} gives the result; in case (S) one uses $\frac1p\sum_{y\ne0}(1-\zeta^{ys})\zeta^{-y\sigma}=[\sigma=0]-[\sigma=s]$.
\end{proof}

\begin{proposition}[necessary condition]\label{prop:necessary}
If $\bar A\cup\bar B\ne[n]$, then $\cC_f$ is not minimal.
\end{proposition}
\begin{proof}
Pick $j\notin\bar A\cup\bar B$. Then $f(x)\ne0$ implies $x_j=0$, so $\supp(c(1,0))$ and $\supp(c(0,e_j))$ are disjoint and $c(1,e_j)$ covers $c(0,e_j)$.
\end{proof}

\begin{proposition}\label{prop:nonminimal}
Let $\bar A\cup\bar B=[n]$. If $m=1$, then $\cC_f$ is not minimal for either notion. If $m=n-1\ge2$ and the notion is saturated, then $\cC_f$ is not minimal.
\end{proposition}
\begin{proof}
If $m=1$, every $x\in I(\GP)\setminus0$ has $x_1\ne0$ (for $x''=0$ because $x\ne0$, for $x''\ne0$ by the definition of $I(\GP)$), so $\supp(c(1,0))\subseteq\supp(c(0,e_1))$, and $c(1,0)$, $c(0,e_1)$ are linearly independent. If $m=n-1$ and the notion is saturated, then $c(1,e_n)$ vanishes at $x\ne0$ only if $x_n\ne0$ (otherwise $x\in\F_p^{n-1}\times0\subseteq I(\GP)$ and $c(1,e_n)_x=1$), hence only at $x=(-1,\dots,-1,x_n)$ with $1+x_n=0$; its unique zero is $(-1,\dots,-1)$, which also lies in $\supp(c(0,e_1-e_2))^c$. Thus $c(1,e_n)$ covers $c(0,e_1-e_2)$, and the two are linearly independent.
\end{proof}

\begin{theorem}\label{thm:main}
Let $p$ be an odd prime, $n\ge4$, $2\le m\le n-2$, $r=n-m$, and $I=[m]_{p-1}\cup V_{p-1}$ with the saturated notion. Then $\cC_f$ is a minimal $[p^n-1,\,n+1,\,p^m+p^r-2]$ code with at most seven nonzero weights (Table~\ref{tab:main}), and it violates the Ashikhmin--Barg condition.
\end{theorem}

\begin{table}[h]
\centering
\begin{tabular}{ll}
\toprule
weight & frequency\\
\midrule
$p^m+p^r-2$ & $p-1$\\
$Q-p^r$ & $(p-1)p^{m-1}$\\
$Q-2$ & $(p-1)(p^{m-1}-1)(p^r-1)$\\
$Q-1$ & $(p-1)(p-2)p^{n-1}$\\
$Q$ & $p^n-1+(p-1)p^{m-1}(p^r-1)$\\
$Q+p^m-2$ & $(p-1)(p^r-1)$\\
$Q+p^r-2$ & $(p-1)(p^{m-1}-1)$\\
\bottomrule
\end{tabular}
\caption{Weight distribution of $\cC_f$ in Theorem~\ref{thm:main}, $Q=(p-1)p^{n-1}$.}
\label{tab:main}
\end{table}

\begin{proof}
\emph{Weights.} Lemma~\ref{lem:Cf-weight} with $b=r$, $\delta=[w=0]$; the seven cases are $v=0$; $v\ne0,\sigma=0$; $\sigma=s$; $\sigma\notin\{0,s\}$, each with $w=0$ or $w\ne0$, together with $u=0$. All weights are positive, so $\dim\cC_f=n+1$.

\emph{Minimality.} We show no nonzero codeword covers another one outside its span. Set $H_u=\{x:u\cdot x=0\}$, $Z(s,u)=\{x\ne0:sf(x)+u\cdot x=0\}$, and note
\begin{equation}\label{eq:size}
|I(\GP)|=p^m+p^r-1\le p^2+p^{n-2}-1<(p-1)p^{n-2}\qquad(2\le m\le n-2,\ n\ge4,\ p\ge3).
\end{equation}
Codewords of $\cC_0=\{c(0,u)\}$ have constant weight $Q$ and none covers another outside its span. Three cases remain.

(i) $a=c(s,u)$, $s\ne0$, $b=c(0,u')$, $u'\ne0$. Covering means $Z(s,u)\subseteq H_{u'}$, and $Z(s,u)\supseteq(H_u\setminus I(\GP))\setminus0$. If $u'\notin\langle u\rangle$ this forces $H_u\setminus H_{u'}\subseteq I(\GP)$, impossible by \eqref{eq:size} since $|H_u\setminus H_{u'}|=(p-1)p^{n-2}$. If $u'=\lambda u\ne0$ it forces $\{x\in I(\GP)\setminus0:u\cdot x=-s\}=\emptyset$; but if $v\ne0$ the equation $v\cdot x'=-s$ has solutions in $\F_p^m\times0\subseteq I(\GP)$, and if $v=0$ then $w\ne0$ and $w\cdot x''=-s$ has a solution $x''\ne0$, giving $(-1,\dots,-1,x'')\in I(\GP)$. If $u=0$ then $|\supp(a)|=|I(\GP)|-1<Q=|\supp(b)|$.

(ii) $a=c(0,u)$, $b=c(s',u')$, $s'\ne0$. Covering means $H_u\setminus0\subseteq Z(s',u')$. On $H_u\setminus I(\GP)$ this says $u'\cdot x=0$; by \eqref{eq:size} this set has more than $p^{n-2}$ points of $H_u$, hence spans $H_u$, so $u'=\lambda u$. Then $s'=0$ on $(H_u\cap I(\GP))\setminus0$, which is nonempty because $H_u\cap(\F_p^m\times0)$ has dimension $\ge m-1\ge1$. Contradiction.

(iii) $a=c(s,u)$, $b=c(s',u')$, $s,s'\ne0$, linearly independent. Scale $b$ so that $s'=s$; then $u\ne u'$ and $a-b=c(0,u-u')$. If $\supp(b)\subseteq\supp(a)$ then $\supp(a-b)\subseteq\supp(a)$, so $a$ covers $c(0,u-u')$, contradicting (i).

\emph{Ashikhmin--Barg.} $w_{\min}=p^m+p^r-2$ and $w_{\max}=Q+p^{\max(m,r)}-2$, the latter with positive frequency since $m\ge2$. Then $w_{\min}/w_{\max}<(p-1)/p$ follows from $p(p^2+p^{n-2}-2)<(p-1)^2p^{n-1}$ for $n\ge4$.
\end{proof}

\begin{example}
$p=3$, $n=4$, $m=2$: $\cC_f$ is an $[80,5,16]_3$ code with weight enumerator $1+2z^{16}+6z^{45}+32z^{52}+54z^{53}+128z^{54}+20z^{61}$; $w_{\min}/w_{\max}=16/61<2/3$. $p=5$, $n=4$, $m=2$: $[624,5,48]_5$, $1+4z^{48}+20z^{475}+384z^{498}+1500z^{499}+1104z^{500}+112z^{523}$.
\end{example}

\begin{remark}\label{rem:MQRT}
Write $W=\F_p^m\times0$ and $W'=0\times\F_p^r$. The defining set of Theorem~\ref{thm:main} is $(W\setminus0)\cup\big((a+W')\setminus\{a\}\big)$ with $a=(-1,\dots,-1,0,\dots,0)$, whereas \cite[Theorem 3.11]{MQRT20} uses $(W\setminus0)\cup(W'\setminus0)$ and proves, for $2\le m<r\le n-2$ and every $p$, that the resulting code is minimal with $w_{\min}/w_{\max}\le(p-1)/p$; the case $m=r$ is covered there only for $p=3$ \cite[Theorem 3.10]{MQRT20}. In both constructions the minimum distance is $|D|=p^m+p^r-2$, the weight of the codeword $f$ itself, so the two families share length, dimension and minimum distance. Their weight distributions differ for every admissible parameter set, so the codes are never equivalent: for $D_0=(W\setminus0)\cup(W'\setminus0)$ the character sum over $W\cup W'$ is $p^m[v=0]+p^r[w=0]-1$, independent of $y$, so the argument of Lemma~\ref{lem:Cf-weight} gives $\wt(c(s,u))=Q-2+p^m[v=0]+p^r[w=0]$ for $s,u\ne0$, and the nonzero weights of $\cC_{f_{D_0}}$ are $p^m+p^r-2$, $Q-2$, $Q$, $Q+p^m-2$, $Q+p^r-2$; the weight $Q-1$, which occurs in Table~\ref{tab:main} with frequency $(p-1)(p-2)p^{n-1}>0$, never occurs. For $p=3$, $n=5$, $m=2$ the two enumerators are $1+2z^{34}+416z^{160}+242z^{162}+52z^{169}+16z^{187}$ (\cite[Example 2]{MQRT20}) and $1+2z^{34}+6z^{135}+104z^{160}+162z^{161}+398z^{162}+52z^{169}+4z^{187}$ (Table~\ref{tab:main}). Theorem~\ref{thm:main} also covers $m=r$ for every $p$.
\end{remark}

\begin{theorem}[unsaturated notion]\label{thm:U}
Let $p$ be an odd prime, $n\ge3$, $2\le m\le n-1$, $r=n-m$, and $I=[m]_{p-1}\cup V_{p-1}$ with the unsaturated notion. Then $\cC_f$ is a minimal $[p^n-1,\,n+1]$ code whose weight distribution is given by Lemma~\ref{lem:Cf-weight}.
\end{theorem}

\begin{proof}
Write $\cC_f=\cC_{D'}$ with $D'=\{(0,x):x\notin I(\GP),x\ne0\}\cup\{(1,x):x\in I(\GP)\setminus0\}$. Here $I(\GP)=(\F_p^m\times0)\cup((\F_p^*)^m\times\F_p^r)$, $N:=\F_p^n\setminus I(\GP)=\{x:x''\ne0,\ x'\text{ has a zero coordinate}\}$, and $S_0:=(\F_p^m\times0)\cap H_u$. The set $N$ spans $\F_p^n$: it contains $(0,e_j)$ for every upper index $j$, and, since $m\ge2$, it contains $(e_i,e_1)$ for every lower index $i$, so $(e_i,0)=(e_i,e_1)-(0,e_1)$ lies in its span. Hence $D'\supseteq(0,N)\cup\{(1,x_0)\}$ spans $\F_p^{n+1}$, and by Lemma~\ref{lem:geom} it suffices that $D'\cap(s,u)^\perp$ spans $(s,u)^\perp$ for every $(s,u)\ne0$.

We use repeatedly: if $v\ne0$, then for every $x''\in\F_p^r$ there is $x'$ with a zero coordinate and $v\cdot x'=-w\cdot x''$ (choose $j$ with $v_j\ne0$, put $x_i=0$ for $i\ne j$ and $x_j=-w\cdot x''/v_j$; this uses $m\ge2$). Hence the $x''$-projection of $H_u\cap N$ is all of $\F_p^r\setminus0$. Also, for $p\ge3$ the set $(\F_p^*)^m$ spans $\F_p^m$, and a subspace minus a proper subspace spans the subspace.

\emph{Case $s=0$, $u\ne0$.} Then $(0,u)^\perp=\F_p\times H_u$ and $D'\cap(0,u)^\perp=\{(0,x):x\in H_u\cap N\}\cup\{(1,x):x\in(H_u\cap I(\GP))\setminus0\}$. Since $\dim S_0\ge m-1\ge1$, the second set is nonempty and contains $S_0\setminus0$; differences of its elements give $0\times S_0$. So the span contains $0\times(S_0+\mathrm{span}(H_u\cap N))$ and some $(1,x_0)$, and it suffices to show $S_0+\mathrm{span}(H_u\cap N)=H_u$. If $v\ne0$, the $x''$-projection identifies $H_u/S_0$ with $\F_p^r$, and $H_u\cap N$ projects onto $\F_p^r\setminus0$. If $v=0$ and $r\ge2$, then $H_u=\F_p^m\times H_w'$ with $H_w'=\{x'':w\cdot x''=0\}\ne0$, $S_0=\F_p^m\times0$, and $H_u\cap N\supseteq\{x'\text{ with a zero coordinate}\}\times(H_w'\setminus0)$ projects onto $H_w'\setminus0$. If $v=0$ and $r=1$, then $H_u=\F_p^m\times0=S_0$.

\emph{Case $s\ne0$.} Scaling, $s=1$. Since $(1,u)^\perp=\{(-u\cdot x,x):x\in\F_p^n\}$ is a graph over $\F_p^n$, it suffices that the $x$-parts $(H_u\cap N)\cup\{x\in I(\GP):u\cdot x=-1\}$ span $\F_p^n$. If $u=0$, then $H_u=\F_p^n$, the second set is empty, and $N$ spans $\F_p^n$ as shown above. If $v\ne0$, the set $\{x'\in\F_p^m:v\cdot x'=-1\}\times0\subseteq I(\GP)$ is an affine hyperplane of $\F_p^m\times0$ not through $0$, hence spans $\F_p^m\times0$, and $H_u\cap N$ projects onto $\F_p^r\setminus0$; together they span $\F_p^n$. If $v=0$ and $r\ge2$, then $H_u\cap N$ spans $\F_p^m\times H_w'$ as above, and $(\F_p^*)^m\times\{x'':w\cdot x''=-1\}\subseteq I(\GP)$ supplies a vector outside $\F_p^m\times H_w'$. If $v=0$ and $r=1$, then $H_u\cap N=\emptyset$, but $\{x\in I(\GP):u\cdot x=-1\}=(\F_p^*)^m\times\{-w_n^{-1}\}$, whose differences span $\F_p^m\times0$ and which contains a vector with nonzero last coordinate.
\end{proof}

\begin{proposition}\label{prop:U-AB}
Under the hypotheses of Theorem~\ref{thm:U}, $\cC_f$ violates the Ashikhmin--Barg condition.
\end{proposition}
\begin{proof}
Put $q=(p-1)/p$ and consider the weights $w_A=|I(\GP)|-1=p^m+(p-1)^m(p^r-1)-1$ ($u=0$), $w_B=Q-1+p^m-(p-1)^m$ ($v=0$, $w\ne0$, of frequency $(p-1)(p^r-1)>0$) and $w_C=Q-1-(p-1)^{m-1}(p^r-1)$ ($k=1$, $w=0$), all realized by Lemma~\ref{lem:Cf-weight}. It suffices to show $p\,w_A<(p-1)w_B$ or $p\,w_C<(p-1)w_B$.

If $m\ge3$, expanding $p\,w_A<(p-1)w_B$ gives the equivalent inequality
\[
p^m+(p-1)^mp^{r+1}-(p-1)^m-1<(p-1)^2p^{n-1}.
\]
Since $(p-1)^mp^{r+1}=(p-1)^2p^{n-1}q^{m-2}$, the right side minus this term is $(p-1)^2p^{n-1}(1-q^{m-2})\ge(p-1)^2p^{n-2}\ge\frac43p^{n-1}>p^m$, using $m\le n-1$ and $(p-1)^2/p\ge4/3$ for $p\ge3$.

If $m=2$, then $p\,w_C<(p-1)w_B$ is equivalent to $Q-1<p(p-1)(p^r-1)+(p-1)(2p-1)=(p-1)(p^{n-1}+p-1)=Q+(p-1)^2$.
\end{proof}

\begin{remark}
The saturated and unsaturated codes have the same length and dimension but different minimum distances; e.g.\ for $p=3$, $n=4$, $m=2$ they are $[80,5,16]_3$ and $[80,5,37]_3$, and for $p=5$, $n=4$, $m=2$ they are $[624,5,48]_5$ and $[624,5,403]_5$.
\end{remark}

\section{The ternary case with arbitrary multiplicities}\label{sec:ternary}

For $p=3$ the restriction to multiplicities in $\{0,p-1\}$ can be removed entirely, because every factor of the generating function simplifies: with $\omega$ a primitive cube root of unity and $a\in\F_3$,
\[
1+\omega^a+\omega^{2a}=3[a=0],\qquad 1+\omega^a=\begin{cases}2,&a=0,\\-\omega^{-a},&a\ne0.\end{cases}
\]
For a multiplicity vector $v\in\{0,1,2\}^k$ put $T_j(v)=\{i:v_i=j\}$, $t_j=|T_j(v)|$, and $\mathrm{box}(v)=\prod_i\{0,\dots,v_i\}\subseteq\F_3^k$. For $u\in\F_3^k$ let $z(u)=|\{i\in T_1(v):u_i=0\}|$, $o(u)=t_1-z(u)$, $s_1(u)=\sum_{i\in T_1(v)}u_i\in\F_3$, $e(u)=[u_{T_2(v)}=0]$, and $\phi(c)=\sum_{y\in\F_3^*}\omega^{-yc}=2[c=0]-[c\ne0]$. Finally let $N_c(o)$ be the number of vectors in $\{1,2\}^o$ with coordinate sum $c\in\F_3$, so that $N_0(o)=(2^o+2(-1)^o)/3$ and $N_c(o)=(2^o-(-1)^o)/3$ for $c\ne0$.

\begin{lemma}\label{lem:ternary-box}
For $u\in\F_3^k$, $\displaystyle\sum_{y\in\F_3^*}\sum_{x\in\mathrm{box}(v)}\omega^{y\,u\cdot x}=3^{t_2}e(u)\,2^{z(u)}(-1)^{o(u)}\,\phi(s_1(u))$.
\end{lemma}
\begin{proof}
The inner sum factors as $\prod_{i\in T_2}(1+\omega^{yu_i}+\omega^{2yu_i})\prod_{i\in T_1}(1+\omega^{yu_i})=3^{t_2}e(u)\,2^{z(u)}(-1)^{o(u)}\omega^{-ys_1(u)}$, and summing over $y$ gives $\phi(s_1(u))$.
\end{proof}

\begin{theorem}[antichain]\label{thm:ternary-antichain}
Let $v\in\{0,1,2\}^n$ with $v\ne0$ and $v\ne(2,\dots,2)$, and let $D=\F_3^n\setminus\mathrm{box}(v)$, so $|D|=3^n-2^{t_1}3^{t_2}>0$. For $u\ne0$,
\[
\wt(c_D(u))=\tfrac13\Big(2|D|+3^{t_2}e(u)\,2^{z(u)}(-1)^{o(u)}\phi(s_1(u))\Big),
\]
and the number of $u\ne0$ with $e(u)=e$, $z(u)=z$, $s_1(u)=c$ is $\big(1\cdot[e=1]+(3^{t_2}-1)[e=0]\big)\binom{t_1}{z}N_c(t_1-z)\,3^{n-t_1-t_2}$, less one for $(e,z,c)=(1,t_1,0)$.
\end{theorem}
\begin{proof}
Lemma~\ref{lem:weight} and Lemma~\ref{lem:ternary-box}. For the count, $u_{T_2}$ is zero or not, $u_{T_1}$ has $z$ zero coordinates and $t_1-z$ coordinates in $\{1,2\}$ with prescribed sum, and the remaining coordinates are free.
\end{proof}

\begin{remark}
Theorem~\ref{thm:ternary-antichain} contains the ternary cases of \cite[Section 4]{HKN19}, which require $t_1+t_2\le2$, and, up to the coordinate scaling $\{0,-1\}=-\{0,1\}$, the codes of \cite[Section 4]{PL22}, which require $t_2=0$ and $t_1\le3$. Since the meet of two boxes is a box, families of several maximal elements are handled by \eqref{eq:incl-excl} with no further computation. Coordinates of multiplicity $1$ make $D$ not closed under scalar multiplication, so these codes are not repetitions of projective codes.
\end{remark}

\begin{theorem}[two levels, saturated]\label{thm:ternary-S}
Let $p=3$, $1\le m\le n-1$, and $I=[m]_2\cup B$ a saturated generalized order ideal of $\GH(m,n)$ in which $B$ has nonzero multiplicity vector $v_B\in\{0,1,2\}^{n-m}$, so that $I(\GP)=(\F_3^m\times0)\cup\big(\{2\}^m\times(\mathrm{box}(v_B)\setminus0)\big)$ and $|D|=3^n-3^m-2^{t_1}3^{t_2}+1$, the $t_j$ referring to $v_B$. For $u=(v,w)\ne0$ with $\sigma=\sum_{i\le m}v_i$,
\[
\wt(c_D(u))=\tfrac13\Big(2|D|+2\cdot3^m[v=0]+3^{t_2}e(w)\,2^{z(w)}(-1)^{o(w)}\phi(\sigma+s_1(w))-\phi(\sigma)\Big).
\]
The number of $u$ in the class determined by $[v=0]$, $\sigma$, $e(w)=e$, $z(w)=z$, $s_1(w)=c$ is the product of the number of $v$ in the class ($1$ for $v=0$; $3^{m-1}-1$ for $v\ne0$, $\sigma=0$; $3^{m-1}$ for $\sigma\ne0$) and $\big([e=1]+(3^{t_2}-1)[e=0]\big)\binom{t_1}{z}N_c(t_1-z)\,3^{n-m-t_1-t_2}$, less one for $u=0$.
\end{theorem}
\begin{proof}
For $y\ne0$ the character sum over $I(\GP)$ is $3^m[v=0]+\omega^{-y\sigma}\big(\sum_{x''\in\mathrm{box}(v_B)}\omega^{y\,w\cdot x''}-1\big)$, and by the proof of Lemma~\ref{lem:ternary-box} the inner sum equals $3^{t_2}e(w)2^{z(w)}(-1)^{o(w)}\omega^{-ys_1(w)}$. Summing over $y$ and applying Lemma~\ref{lem:weight} gives the formula; the counts are as in Theorem~\ref{thm:ternary-antichain}.
\end{proof}

The weight depends on $w$ only through $(e,z,s_1)$ and on $v$ only through $[v=0]$ and $\sigma$, so the number of nonzero weights is bounded by a function of $t_1$ alone. The linear condition $\sigma=0$ of Theorem~\ref{thm:complement-S} is replaced by the two conditions $\sigma=0$ and $\sigma+s_1(w)=0$, in which the upper-level coordinates of multiplicity one enter through their sum.

\begin{proposition}\label{prop:ternary-minimal}
If $2\le m\le n-2$, the codes of Theorem~\ref{thm:ternary-S} satisfy the Ashikhmin--Barg condition and are therefore minimal. If $m=n-1$ they are not minimal.
\end{proposition}
\begin{proof}
Every nonzero codeword of $\cC_D$ has weight at most $|\{x:u\cdot x\ne0\}|=2\cdot3^{n-1}$. In the formula of Theorem~\ref{thm:ternary-S} the term $3^{t_2}e\,2^{z}(-1)^{o}\phi(\sigma+s_1)$ is at least $-3^{t_2}2^{t_1}=-|\mathrm{box}(v_B)|$ (if $o=0$ it is at least $-3^{t_2}2^{t_1}$, and if $o\ge1$ its absolute value is at most $2\cdot3^{t_2}2^{t_1-1}$), and $-\phi(\sigma)\ge-2$. Hence $3\,w_{\min}\ge2|D|-|\mathrm{box}(v_B)|-2=2\cdot3^n-2\cdot3^m-3|\mathrm{box}(v_B)|$, and the Ashikhmin--Barg inequality $3w_{\min}>2w_{\max}$ follows from $2\cdot3^{n-1}>2\cdot3^m+3^{n-m+1}\ge2\cdot3^m+3|\mathrm{box}(v_B)|$. The function $m\mapsto2\cdot3^m+3^{n-m+1}$ is convex, so it suffices to check $m=2$ and $m=n-2$, where the inequality reads $3^{n-1}>18$ and $4\cdot3^{n-2}>27$; both hold for $n\ge4$. For $m=n-1$ the vector $v_B$ has weight one and the claim is part of Theorem~\ref{cor:griesmer}.
\end{proof}

\begin{example}
(a) $n=4$, $v=(1,1,1,0)$: Theorem~\ref{thm:ternary-antichain} gives a $[73,4,48]_3$ code with weight enumerator $1+24z^{48}+18z^{49}+36z^{50}+2z^{54}$, in agreement with \cite[Theorem 4.3]{PL22} for $n=4$.
(b) $n=4$, $m=2$, $v_B=(1,1)$: Theorem~\ref{thm:ternary-S} gives a $[69,4,45]_3$ code with weight enumerator $1+22z^{45}+30z^{46}+18z^{47}+2z^{48}+2z^{51}+6z^{52}$; it is minimal by Proposition~\ref{prop:ternary-minimal}, and $w_{\min}/w_{\max}=45/52$.
\end{example}

Within this family the Griesmer bound is attained exactly when $v_B$ has weight one, which is the case $p=3$ of Theorem~\ref{cor:griesmer}; we found no further Griesmer codes for $n\le5$. The ternary analogues of the codes $\cC_f$ of Section~\ref{sec:Cf}, whose weights follow from the same character sums, are left for future work.

\section{Concluding remarks}\label{sec:conclusion}

\subsection*{Intermediate multiplicities for $p\ge5$}
For a generalized order ideal with multiplicity vector $(v_1,\dots,v_n)$ on an antichain, Lemma~\ref{lem:weight} reduces the weight of $c_D(u)$ to
\[
\sum_{y\in\F_p^*}\prod_{i}\Big(\sum_{j=0}^{v_i}\zeta^{y u_i j}\Big)=p\cdot\big|\{x\in\textstyle\prod_i[0,v_i]:u\cdot x=0\}\big|-\prod_i(v_i+1),
\]
so the weight distribution is equivalent to counting solutions of a linear equation in a box. This is trivial when every $v_i\in\{0,p-1\}$, and also when a single coordinate carries an arbitrary multiplicity (Proposition~\ref{prop:antichain-t}, Theorem~\ref{cor:griesmer}), since a one-coordinate box is a set of $t+1$ points. With several intermediate coordinates it is tractable only when the multiplicities are close to $0$ or to $p-1$, because $\sum_{j\le v}\zeta^{ja}=-\sum_{j>v}\zeta^{ja}$ for $a\ne0$ turns a factor with $v_i=p-2$ into $-\zeta^{-u_iy}$ and a factor with $v_i=p-3$ into $-(\zeta^{-u_iy}+\zeta^{-2u_iy})$. This is precisely why the down-sets treated in \cite{HKN19,MHL26} are generated by $(r,0,\dots,0)$, $(1,r,0,\dots,0)$, $(2,r,0,\dots,0)$, $(p-2,r,0,\dots,0)$ and $(p-3,r,0,\dots,0)$, and why the number of terms grows with the number of coordinates carrying an intermediate multiplicity. For $p=3$ every multiplicity is of this kind, which is what makes Section~\ref{sec:ternary} possible; for $p\ge5$ the hierarchical case with two or more intermediate multiplicities on the upper level remains open, as does the corresponding question for down-sets raised in \cite{MHL26}.

\subsection*{Two order ideals and more levels}
Families of two incomparable generalized order ideals $[m]_{p-1}\cup(\bar B_1)_{p-1}$, $[m]_{p-1}\cup(\bar B_2)_{p-1}$ are handled by \eqref{eq:incl-excl} with no new character sums, since the intersection is again of the same form. In our computations the complement codes of such families were always minimal and satisfied the Ashikhmin--Barg condition, so we did not pursue them. Hierarchical posets with three or more levels lead, in the saturated notion, to defining sets built from nested blocks $\{-1\}^{m_1}\times\{-1\}^{m_2}\times\cdots$, and the computation of Lemma~\ref{lem:charsum} goes through level by level; we expect the number of weights to stay bounded in terms of the number of levels.

\subsection*{Optimality of $\cC_f$}
The codes of Theorems~\ref{thm:main} and~\ref{thm:U} are far from the Griesmer bound; for $p=3$, $n=4$, $m=2$ the Griesmer bound allows $d\le53$ for an $[80,5]_3$ code, against $d=16$ and $d=37$ respectively. As in the binary case \cite[Theorem 6.3]{HKWY20}, their interest lies in minimality outside the Ashikhmin--Barg range, and in the fact that the saturated and unsaturated notions give codes of the same length and dimension but different minimum distances. Whether the corresponding secret sharing schemes have access structures that can be described in terms of the poset, as in \cite{HKWY20}, is a natural question that we leave open.

\end{document}